\documentclass[preprint,12pt]{elsarticle}
\usepackage[width=18cm]{geometry}
\usepackage{amssymb}
\usepackage{amsmath}
\usepackage[utf8]{inputenc}
\usepackage{mathtools,amsthm,amsmath}
\usepackage{comment}
\newtheorem{theorem}{Theorem}
\newtheorem{remark}{Remark}
\usepackage{color}
\usepackage{marginnote}
\newcommand{\R}{\mathbb R}
\newcommand{\N}{\mathbb N}
\newcommand{\abs}[1]{\left\lvert #1 \right\rvert}
\newcommand{\ip}[2]{\left\langle#1,#2\right\rangle}
\newcommand{\iip}[2]{\left(#1,#2\right)}

\newcommand{\eps}{\varepsilon}

\newcommand{\spec}{\operatorname{spec}}
\def \p{\partial}
\DeclareMathOperator{\Div}{div}

\def\beq{\begin{equation} } \def\eeq{\end{equation}}

\def\eps{\varepsilon}  

\def\ben{\begin{enumerate} }

\def\een{\end{enumerate} }

\def \p { \partial}

\begin{document}
\begin{frontmatter}

\title{Principal spectral rigidity implies subprincipal spectral rigidity}
\date{\today}
\author[2]{Maarten V. de Hoop}
 \ead{mvd2@rice.edu}
 \address[2]{Rice University, Houston, Texas, USA}

 \author[3]{Joonas Ilmavirta}
 \address[3]{University of Jyv\"{a}skyl\"{a}, Jyv\"{a}skyl\"{a}, Finland}
 \ead{joonas.ilmavirta@jyu.fi}

\author[4]{Vitaly Katsnelson\fnref{label2}}
\address[4]{New York Tech, New York, NY, USA}
 \ead{vkatsnel@nyit.edu}
 \fntext[label2]{Corresponding Author}
 

\begin{abstract}
We study the inverse spectral problem of jointly recovering a radially symmetric Riemannian metric and an additional coefficient from the Dirichlet spectrum of a perturbed Laplace-Beltrami operator on a bounded domain. Specifically, we consider the elliptic operator  
\[
L_{a,b} := e^{a-b} \nabla \cdot e^b \nabla
\]
on the unit ball \( B \subset \mathbb{R}^3 \), where the scalar functions \( a = a(|x|) \) and \( b = b(|x|) \) are spherically symmetric and satisfy certain geometric conditions. While the function \( a \) influences the principal symbol of \( L \), the function \( b \) appears in its first-order terms. We investigate the extent to which the Dirichlet eigenvalues of \( L_{a,b} \) uniquely determine the pair \( (a, b) \) and establish spectral rigidity results under suitable assumptions. 
\end{abstract}


\end{frontmatter}

\section{Introduction}

We investigate the joint recovery of a (radially symmetric Riemannian) metric and an additional coefficient from a single spectrum of an elliptic operator $L$ in a rigidity sense on a manifold with boundary. The operator under consideration is a perturbed Laplace Beltrami operator and has the form:\[
L= L_{a,b}:= e^{a-b}\nabla \cdot e^b \nabla,
\]
for spherically symmetric functions $a = a(|x|), b = b(|x|)$ satisfying certain assumption described later. Note that only $a$ appears in the principal symbol of $L$ while $b$ appears in the first order constituent of $L$.
The boundary conditions are Dirichlet, and the domain is the unit ball $B \subset \mathbb R^3$. Everything will be radial, so the two scalar functions $a$ and $b$ will only depend on $r=\abs{x}$. The proof works for Neumann/Robin boundary conditions with an extra assumption described in the main theorem.
\begin{remark}
Our setup and results include a number of important settings.
In the acoustic wave setting for fluid regions, the pressure field $p(t,x)$ satisfies the wave equation
 \[ - \kappa\nabla \cdot (\rho^{-1}\nabla p) = \omega^2 p
 \]
 where $\kappa>0$ is the bulk modulus and $\rho$ is the density. The natural boundary condition is $p|_{\p B} = 0$. Hence, our results apply to the acoustic wave equation where $e^b = \rho^{-1}$ and $e^{a-b} = \kappa$, so that $e^a = c^2 = \kappa/\rho.$
Also, as shown in \cite{HIKRigidity}, in a spherically symmetry Earth model, the radial constituent of toroidal modes must satisfy the scalar equation of the form
\[
\rho^{-1}
       (-\nabla \cdot \mu \nabla + r^{-1}(\p_r \mu)) v = \omega^2 v
\]
where $\mu>0$ is a Lam\'{e} coefficient and $\rho$ is the density, subject to a natural Robin type boundary condition $\mu \p_r v|_{\p B} = r^{-1}\mu \ v|_{\p B}$. 
With $a,b$ defined by $e^b = \mu$ and $e^{a-b} = \rho^{-1}$, so that $e^a = c_S^2 = \mu/\rho.$ This operator becomes $e^{a-b}\nabla \cdot e^b \nabla - r^{-1}e^{a-b}(\p_r e^b)$ which is the operator we consider here modulo a zeroth order term. With a slight modification to the proof, our results apply to this operator as well.
\end{remark}

\emph{In our previous work \cite{HIKRigidity}, we demonstrated rigidity of the wave speed (corresponding to the coefficient $a$) but do not address the recovery of the density of mass (associated to coefficient $b$). The primary challenge lies in the fact that $b$ does not appear in the principal symbol of the operator $L$, rendering standard methods using the length spectrum inapplicable.
To the best of our knowledge, this is the first result of its kind concerning a manifold with boundary, where we recover multiple coefficients from a single spectrum (in a rigidity sense) while the principal symbol remains independent of one of the coefficients. To recover an additional coefficient, we will present a novel application of Kato perturbation techniques in inverse spectral problems to recover an additional coefficient not in the principal symbol of the operator.} See \ref{a: history} for a history of this problem.

We require the so-called Herglotz condition on the wave speed while allowing an unsigned curvature; that is,
curvature can be everywhere positive or it can change sign, and we
allow for conjugate points. 
Our manifold is the Euclidean ball $M = \bar B(0,1)\subset \mathbb{R}^3$, with the metric $g(x) =
e^{-a(|x|)} e(x)$, where $e$ is the standard Euclidean metric and
$a \colon (0,1]_r \to (0,\infty)$ is a function satisfying suitable
conditions, where $r = |x|$ is the radial coordinate.
We work in dimension three, but the results may be generalized to higher dimension.

We also need an assumption on the degeneracy of the spectrum of $L_0:= L_{a_0,b_0}$ corresponding to the coefficients $a_0, b_0$. As is well known, the spherical symmetry creates natural degenerecies in the spectrum, and we will assume that those are the only degenerecies in the spectrum. We also use the terminology from \cite{HZellipse} that when we have a family of functions $b_s$ that depend $C^\infty$ smoothly on $s$, then we say $\{b_s\}$ is \emph{flat} if its Taylor series at $s =0$ vanishes.

The main theorem we will prove is (see \cite{HIKRigidity} for the definitions of \emph{countable conjugacy condition}, \emph{clean intersection hypothesis}, and \emph{geometric spreading injectivity condition})
\begin{theorem} \label{t: intro main theorem acoustic)}
Let $B=\bar B(0,1)\setminus \bar B(0,R)\subset\R^3$, $R\geq0$, be an
annulus (or a ball if $R=0$).  Fix $\eps>0$ and let $e^{a_s/2}$,
$s\in(-\eps,\eps)$, be a $C^\infty$ function $[R,1]\to(0,\infty)$
satisfying the Herglotz condition and the countable conjugacy
condition and depending $C^\infty$-smoothly on the parameter $s$.
Assume also that the length spectrum satisfies the geometric spreading injectivity condition, and assume that the periodic broken rays satisfy the clean intersection hypothesis.
If
$R=0$, we assume that all odd order derivatives of $e^{a_s/2}$ vanish at
$0$. 
Assume $b_s(r)$ is a $C^\infty$ function and $b_s(1)$ is independent of $s$.
If $\text{spec}(L_s) = \text{spec}(L_{a_s,b_s})$ is the same for all $s$, then $a_s = a_0$ for each $s$, and $b_s$ is flat.
Moreover, the same result holds with Neumann/Robin boundary conditions under the assumption that 
$\p_r b_s(1)$ is independent of $s$. 
\end{theorem}

The assumptions in the above theorem are natural for spectral inverse problems. Roughly speaking, the Herglotz condition ensures that rays do not become trapped, providing sufficient data for recovery. The geometric spreading injectivity condition serves as an analog to having a simple length spectrum, which prevents cancellations between different rays in a wave trace formula. The clean intersection hypothesis allows for recovery of the length spectrum from the spectral data, and the countable conjugacy condition avoids certain degenerate or anomalous cases in the length spectrum. Exact definitions can be found in \cite{HIKRigidity}. In the non-Dirichlet boundary case, we have an extra assumption that can be found in past inverse spectral works such as \cite{Sini1dSpectral}. As a crucial component to proving these theorems, we also investigate the squares of eigenfunctions associated with an elliptic partial differential operator, which are natural to study in other contexts to recover information of lower order coefficients (see \cite{YannickSquaredEfunctions} and \ref{a: history}).

Let us briefly describe the organization and proof of the theorem. In the first section, we verify the our result \cite[Theorem 1.3]{HIKRigidity} on spectral rigidity of the wave speed continues to hold when the density varies as well. Next, we apply this theorem to show the ``density of squares" of eigenfunctions for radial functions. Finally, we use Kato-type perturbation theory to show rigidity of the $b$, where we view the $b$ term of the operator $L_{a,b}$ as a perturbation of a Laplacian associated to a conformally Euclidean metric.

\section{Leading order spectral rigidity}

In previous work, we proved the following theorem that is reformulated using our notation here (see \cite[Theorem 1.3]{HIKRigidity} )
\begin{theorem} \label{thm: pld spec rigidity}
Let $B=\bar B(0,1)\setminus \bar B(0,R)\subset\R^3$, $R\geq0$, be an
annulus (or a ball if $R=0$).  Fix $\eps>0$ and let $a_s$,
$s\in(-\eps,\eps)$, be a $C^\infty$ function $[R,1]\to(0,\infty)$
with $e^{a_s/2}$
satisfying the Herglotz condition and the countable conjugacy
condition and depending $C^1$-smoothly on the parameter $s$.
Assume also that the length spectrum satisfies the geometric spreading injectivity condition, and assume that the periodic broken rays satisfy the clean intersection hypothesis.
If
$R=0$, we assume that all odd order derivatives of $e^{a_s}$ vanish at
$0$.  If $spec(L_{a_s,b})=spec(L_{a_0,b})$ for all $s\in(-\eps,\eps)$,
then $a_s=a_0$ for all $s\in(-\eps,\eps)$.
\end{theorem}
The proof of the theorem only relied on the principal symbol of the elliptic operator $ {e^{a-b}} \nabla \cdot e^{b} \nabla$, which is the same as that of the Laplace-Beltrami operator for the metric $e^a dx$. Hence, the theorem continues to hold if $b$ is allowed to vary as well with $s \in (-\eps, \eps)$. We now consider a variation $(a_s,b_s)$; if $ e^{a_s/2} = c_s$ satisfies all the assumptions in Theorem \ref{thm: pld spec rigidity}, then we say that $(a_s,b_s)$ are \emph{admissable} radial profiles. For such profiles, the above theorem continues to hold and we have

\begin{theorem}
\label{thm:LO}
Let $a_s(r)$ and $b_s(r)$ be two admissable radial profiles.
If $\spec(L_{a_s,b_s}) = \spec(L_{a_0,b_0})$ for all $s \in (\eps,\eps)$, then $a_s$ is independent of $s$.
\end{theorem}
\begin{proof}
This follows directly from the proof of Theorem \ref{thm: pld spec rigidity} in \cite{HIKRigidity}. The proof relied on two components. The first was a trace formula that recovers the length spectrum from the spectrum. The operator considered in that proof had the density $\rho$ included as well for greater generality so the trace formula  \cite[Proposition 14]{HIKRigidity} holds for $L_{a,b}$. The length spectrum is independent of density so that $\text{lsp}(L_{a_s,b_s}) = \text{lsp}(L_{a_s,b})$. Thus, we can recover the length spectrum from the spectrum with the same argument there. The rest of the proof only relied on the length spectrum to recover $a$, and did not depend on the density. Hence, $a_s$ is independent of $s$.
\end{proof}

We can extend Theorem \ref{thm:LO} to show that if the ``derivatives of the spectrum vanish'', then so do the $s$-derivatives of $a_s$ at $s = 0$.
Denote $\delta^k = \frac{d^k}{ds^k}|_{s=0}$ as the operator which takes derivatives in $s$ and then restricts to $s=0$. If $f_s$ denotes a family of distributions or constants, it will be convenient to denote
\[
f' : = \delta f \qquad f'' := \delta^2 f
\]
where we left out the subscript $s$.

Denote $\Lambda_s:= \{\lambda_s^1, \lambda_s^2, \dots \}$ be the spectrum of $L_s$ and
$\Lambda'$ as the first order perturbation. That is, the collection $\delta \lambda^k = \frac{d}{ds}|_{s=0} \lambda^k_s$ for each $k$. 
We have the following corollary which is proved in \ref{app:thm4}.
\begin{theorem}\label{t: derivative rigidity}(corollary from proof of above theorem)
\label{col:LO}
Let $a_s(r)$ and $b_s(r)$ be two radial profiles, such that $(a_0(r), b_0(r))$ are admissable.
If $\Lambda' =\{ 0\}$, then $a'=0$.
\\
\end{theorem}
The proof of the above theorem would be automatic from previous results if $(a_s,b_s)$ was admissable. However, that is too strong of an assumption, but the result still remains true and requires a careful argument that we provide in the appendix that uses a technique in \cite{HZellipse}.
Notice that no assumptions are needed on $b$.
We only needed to make assumptions on the principal behaviour to prove the theorem.

\subsection*{Density of squared eigenfunctions}

We can use the previous theorem to show that the squares of eigenfunctions of $L$ have a certain density property among smooth radial functions. We need several preliminaries to prove such a theorem. Let $a_s(r)$ be a smooth variation corresponding to wave speeds $c_s$, and keep $b$ fixed.
Denote $a'_s(r)\coloneqq\partial_sa_s(r)$ and consider the family of operators $L_s = L_{a_s,b}$. Recall, the eigenvalues of $L_s$ are denoted $\lambda_s^k$ with corresponding eigenspaces $E_s^k\subset C^\infty(B)$, indexed with $k\in\N$. Denote $L = L_{a_0,b}$, with eigenfunctions $\{\psi_j\}$, and let $\psi$ denote a generic eigenfunction.
We will now compute $\Lambda'$.

 The operators $L_s$ are self-adjoint with respect to the metric $e^{b-a_s}dx$. Let $\phi_s$ be a family of eigenfunctions of $L_s$ depending on the parameter $s$. We have with $L_s \phi_s = \lambda_ s \phi_s$ and differentiating both sides of this equation with respect to $s$ and setting $s =0$ (with $\lambda:=\lambda_0, \phi := \phi_0)$ gives
\[
a' L \phi + L\phi '= \lambda'\phi + \lambda \phi',
\]
so
\[
a' \lambda \phi + L\phi'= \lambda'\phi + \lambda \phi'.
\]

Pair both sides of the above equation with $\phi$ using the inner product with the volume form $e^{b-a} dx$, and using the self-adjointness of $L$ we compute
\[
\lambda \langle a' \phi, \phi \rangle
= \lambda ' ||\phi||^2
= \lambda'
\]
if $\phi$ is normalized. In the above formula we used our assumption that the degeneracy of the spectrum of $P_0$ is only due to spherical symmetry. In general, if $\Pi_\lambda$ denotes the orthogonal projection to the $\lambda$ eigenspace, then $\lambda'$ is an eigenvalue of $\Pi_\lambda L' \Pi_\lambda^*$ and $\phi$ would need to be chosen more carefully to correspond with eigenvectors of this matrix. However, due to the spherical symmetry, this matrix is a multiple of the identity so $\phi$ can be any $\lambda$ eigenfunction.
Hence, we derive
\begin{equation}\label{e: lambda'_k formula}
\delta \lambda^k
= \delta \lambda^k \int a'(r) |\psi_k(r)|^2 e^{b-a}\ dx
:= \delta \lambda^k
\iip{a'}{\abs{\psi_k}^2},
\end{equation}
where we have defined the inner product $\iip{\cdot}{\cdot}$ in the last equality.

We can now prove the following result for the density of squares of radial eigenfunctions:

\begin{theorem}
\label{thm:density-of-squares}
If $f\colon B\to\R$ is a smooth radial function so that
\begin{equation}\label{e: orthogonal to squares}
\iip{f}{\abs{\psi}^2}
=
0
\end{equation}
for all eigenfunctions $\psi$ of $L$, then $f=0$.
\end{theorem}

\begin{proof}
Let $a_s = a + sf$. First, \eqref{e: orthogonal to squares} implies by our computation \eqref{e: lambda'_k formula} for $\delta \lambda^k$ that $\Lambda' = \{0\}$. By theorem \ref{t: derivative rigidity}, $f  = 0$.
\end{proof}

\section{Proof of Theorem \ref{t: intro main theorem acoustic)}}

In order to prove Theorem \ref{t: intro main theorem acoustic)}, 
we will apply this density of squares to obtain the spectral rigidity of $b_s$.
(We assume that $b_s$ is known on the boundary at $r=1$).

\begin{proof}[Proof of Theorem \ref{t: intro main theorem acoustic)}]
Let $(a_s(r)$ and $b_s(r))$ be two admissal radial profiles.
For the Neumann spectrum, additionally assume that $\p_r b_s(1)$ is independent of $s$.
By the assumption that the spectrum of $L_{a_s,b_s}$ is the same for all $s$, then $a_s$, 
we apply theorem~\ref{thm:LO} to conclude that $a_s=a_0$ for all $s$.

Let us then consider the perturbed operator $L'=\partial_s|_{s=0} L_{a,b_s}$, which now only depends on the perturbation $b'$.
It operates as
\begin{equation}
L'u
=
e^a\ip{\nabla b'}{\nabla u}.
\end{equation}
Note that $\Lambda'=\{0\}$ since $\text{spec}(L_s)$ is independent of $s$. Spectral perturbation theory says that the spectrum of $L'$ has zero derivative if and only if
\begin{equation}
0
=
\iip{\psi}{L'\psi}
=
\int_B \psi^*e^a\ip{\nabla b'}{\nabla\psi} e^{b-a} \ dx
=
\frac12\int_B e^b\ip{\nabla b'}{\nabla\abs{\psi}^2} \ dx
\end{equation}
for all eigenfunctions $\psi$ of $L$.
Integrating by parts and using that $\psi$ vanishes at the boundary in the Dirichlet case or using that $\p_r b'(1) = 0$ in the Neumann case, we find
\begin{equation}
\int_B \abs{\psi}^2\Div(e^b\nabla b') \ dx
=
0.
\end{equation}

By theorem~\ref{thm:density-of-squares}, we obtain
\begin{equation}
\Div(e^b\nabla b')
=
0.
\end{equation}
Thus, 
\[
0 = \int_B \Div(e^b \nabla b') b' \ dx
= -\int_B e^b |\nabla b'|^2 \ dx. 
\]
Since $b'$ is radial, we conclude $\p_r b' = 0$, so that $b' \equiv 0$ since we assumed $b'(1) =0$.

 For the second order, we similarly have $\lambda '' = 0$.
 Since we showed $\p_r b' =0$, we conclude that $L' = 0$ so that we get the expression
 $\lambda'' = \langle L'' \psi , \psi \rangle$
 where $L'' = \frac{d^2}{ds^2}|_{s=0}L_{a_0,b_s}$. Usually, the second order perturbation $\lambda''$ has a more complicated form involving $L'$, but since we already showed $L' =0 $, we get this simple expression for $\lambda''$.
 Hence, we get
 \[
 \langle \psi, L'' \psi \rangle = 0,
 \]
 where since $a_s$ is independent of $s$, 
 \[
 L'' = e^a \nabla b'' \cdot \nabla.
 \]
Hence, the same argument shows $\p_r b'' =0$. Continuing, we obtain $\p_r b^{(j)} = 0$ for each $j$. \end{proof}

\subsection*{Acknowledgements
} MVdH was supported by the Simons Foundation under the MATH + X program, the National Science Foundation under grant DMS-1815143, and the corporate members of the Geo-Mathematical Imaging Group at Rice University.
JI was supported by the Academy of Finland (projects 332890 and 336254).

\section{Declarations}
\subsection*{Funding}
MVdH was supported by the Simons Foundation under the MATH + X program, the National Science Foundation under grant DMS-1815143, and the corporate members of the Geo-Mathematical Imaging Group at Rice University.
JI was supported by the Research Council of Finland (Flagship of Advanced Mathematics for Sensing Imaging and Modelling grant 359208; Centre of Excellence of Inverse Modelling and Imaging grant 353092; and other grants 351665, 351656, 358047) and the V\"ais\"al\"a project grant by the Finnish Academy of Science and Letters.

\subsection*{Conflict of interest/Competing interests}
{\bf Financial interests:} The authors declare they have no financial interests.
\\

\noindent {\bf Non-financial interests:} The authors declare they have no non-financial interests.

\subsection*{Availability of data and material} Not applicable

\subsection*{ Code availability} Not applicable

\appendix

\section{History of the problem}\label{a: history}
The general problem of finding a manifold from spectral data is old; see e.g. the famous question~\cite{Kac66} by Mark Kac in 1966.
In 1972 he showed that the spectrum of a rotational solid allows one to determine lateral surface area~\cite{Kac72}.
Another classic result is that of Guillemin and Kazhdan~\cite{GUILLEMINtwomanifold} connecting spectral rigidity with periodic ray transforms on negatively curved manifolds.
One approach to spectral problems of this nature is Weyl's law~\cite{Ivrii100years} and its many variants.
Semiclassical variants of the law have proven useful for inverse problems since recovering the metric or a potential from a parameterized spectrum associated to a semiclassical operator depending on a parameter $h$ is much simpler than from a single spectrum, although the potentials are allowed to be more complicated; 
see e.g. \cite{Verdire2011ASI, VerdireAIF_2007__57_7_2429_0,Gurarie}.
For a more detailed description of progress on similar problems to recover a single parameter from spectral data, we refer the readers to~\cite{Gordon2005, GordonSurvey,DatchevSurvey}.

On manifolds with boundary or ones with symmetry in the metric, there are many past results on recovering a single coefficient. See for example \cite{Bru84,stroock1973, zelditch1998,Gurarie,HZellipse,zelditch1998}.
Fewer results exist on recovering multiple coefficients from a single spectrum. 
A particularly relevant paper is by Barcilon in \cite{BarcilonSpheroidal}. Barcilon shows how to reconstruct $n$ coefficients of a $n+1$ order ordinary differential operator, but it requires $n+1$ distinct spectra associated with $n+1$ distinct boundary conditions. It is meant to illustrate the subtlety of recovering the density, bulk modulus, and shear modulus from spheroidal modes. 

A key result in this letter was showing the squares of eigenfunctions form a dense set within a particular $C^\infty$ function space. 
 As described in \cite{YannickSquaredEfunctions},
squared eigenfunctions appear quite naturally in
the study of the controllability of the bilinear Schr\"odinger equation. The $k$th eigenvalue $\lambda_k^\epsilon$ of $\Delta + \epsilon V: H^2(\Omega) \cap H^1_0(\Omega) \to L^2(\Omega)$ is analytic with respect to $\epsilon$ and satisfies
\[
\frac{d}{d\epsilon}|_{\epsilon = 0} \lambda_k^\epsilon = \int V(x) \phi_k(x)^2 dx
\]
where $\phi_k$ is the corresponding eigenfunction for $-\Delta$. Hence, the linear independence and density of the $(\phi^2_n)_{\{n \in \mathbb N\}}$ within a certain function space clearly plays a role in the study of the
size of the family of potentials $V$ for which the spectrum has some prescribed
property. An additional novelty of this work is we show that the squares of eigenfunctions reveal information about the non-principal coefficient $b$, and are connected to the geometry of $(B, e^a dx)$ via the length spectrum.

\section{Proof of Theorem \ref{t: derivative rigidity}}
\label{app:thm4}

\begin{proof}
Since the proof involves a wave propagator, let us write $c_s = e^{a_s/2}$ to represent the wave speed, and $\rho_s = e^{b_s-a_s}$.
Let us first show the formal proof assuming each $c_s$ is admissable, and then do the rigorous justification when only $c_0$ is admissable. Let $G_s$ denote the Green's function for the wave equation with Neumann boundary conditions, and we have by definition, the distribution in the variable $t$
\begin{equation}\label{e: tr(gs)}
\text{Tr}(G_s)(t) = \sum_k \cos(t\sqrt{-\lambda_s^k} ),
\end{equation}
which implies
\[
\delta \text{Tr}(G_s)(t)
= \sum_k \frac{\delta \lambda^k}{2\sqrt{-\lambda_0^k}}t\sin( t\sqrt{-\lambda_0^k}),
\]
and our assumption implies this quantity is $0$.
On the other hand, for a differentiable family $T_s$ in the length spectrum ($\text{lsp}_s$) and $J$ a neighborhood of $T_s$ isolated from the rest of the length spectrum, the trace formula of \cite[Proposition 2.3]{HIKRigidity}, \emph{if} each $c_s$ was admissable, gives
\begin{equation}\label{e: tr(G) trace formula}
\text{Tr}(G_s)|_J
= a_s (t-T_s+i0)^{-5/2} + O((t-T_s+i0)^{-3/2}),
\end{equation}
where $a_s$ is independent of $t$. The differentiable families $\{ T_s\}$ in the length spectrum are constructed in \cite[section 4]{HIKRigidity}. The geometric spreading injectivity condition guarantee that $a_s$ does not vanish.
Taking a first order perturbation gives
\begin{equation}\label{e: delta Tr(G) formula}
0 = \delta\text{Tr}(G)|_J
= a_0 (-d)(t-T_0+i0)^{-7/2}\delta T + O((t-T_0+i0)^{-5/2}).
\end{equation}
Hence, we conclude that $\delta T = 0$.

Next, let $\phi\colon(-\delta,\delta)\to(R,1)$ for any $\delta>0$ be a $C^1$ function
satisfying $\phi(s)\in P_s$ for all $s\in(-\delta,\delta)$, where $P_s$ is the set of radii corresponding to a periodic broken geodesic.
Let $\gamma_0: [0,T_0] \to M$ be the periodic broken geodesic of length $T_0$ and radius $\phi(0)$. Then $T_s = T(\phi(s))$ is a $C^1$ family of lengths of periodic orbits and we have from \cite[Lemma 4.13]{HIKRigidity}
that
\[
2\delta T
= \int_0^{T_0}
\frac{d}{ds}c_s^{-2}(\gamma_0(t))\Large|_{s=0} \ dt.
\]
If $f\colon M\to\R$ is a continuous radially symmetric function (identified as a
function $f\colon(R,1]\to\R$) and $r\in P_0$, then the integral of $f$ over any
periodic geodesic (with respect to sound speed $c_0$) of radius $r$ is an integer multiple of
\begin{equation}\label{eq:abel-transform}
\int_r^1\frac{f(r')}{c(r')}\left(1-\left(\frac{rc(r')}{r'c(r)}\right)^2\right)^{-1/2}\ d
r'.
\end{equation}
\\ \\
Since we showed $\delta T = 0$, this implies that the variation of the wave
speed, $f=\left.\frac{d}{ds}c^{-2}_s\right|_{s=0}\colon M\to\R$, integrates to zero
over all periodic geodesics of radius $r$.
This function $f$ is radially symmetric, so we can think of it as a function
$(R,1]\to\R$.
By \eqref{eq:abel-transform} we know that
\begin{equation}
\label{eq:abel=0}
\int_r^1\frac{f(r')}{c(r')}\left(1-\left(\frac{rc(r')}{r'c(r)}\right)^2\right)^{-1/2} \ dr'
=
0.
\end{equation}
Equation~\eqref{eq:abel=0} is true for a dense set of radii
$r\in(R,1)$ as shown in \cite[proof of Theorem 4.6]{HIKRigidity}, so it follows from
\cite[lemma 4.14]{HIKRigidity} that in fact $f$ vanishes identically.
We have found that $\delta c =0$ so $\delta a = 0$.

The above proof works if each member of the family $c_s$ was admissable. Let us modify the proof above to show that we get the same formulas even if only $(c_0,\rho_0)$ are admissable using an argument in \cite{HZellipse}. It will be convenient to relabel the eigenvalues as $-\lambda^2_j(s)$ where $L_s \Psi_j = -\lambda_j^2 \Psi_j.$
Recall that
\[
\text{Tr} \ \p_t G_s(t) = \sum_j \cos(t \lambda_j(s)),
\]
so by the analogous computation in \cite{HZellipse}, we have
\[
\frac{d}{ds}|_{s=0}
\text{Tr} \ \p_t G_s(t)
= -t \sum_j \frac{\sin(t\lambda_j)}{2\lambda_j} \lambda_j'
= -\frac{t}{2} \sum_j \int_B \frac{\sin(t\lambda_j)}{\lambda_j}  c'(x) |\Psi_j(x)|^2 \rho(x)\ dx
\]
where we note that besides $c'$, each quantity above is associated to the unperturbed operator $P_0$. Notice that the right hand side is exactly the trace of 
\[\tilde G(t,x,x_0) = -\frac{t}{2} \text{Re}\sum_j \frac{e^{i\lambda_j t}}{i\lambda_j} c'(x)\Psi_j(x)\Psi_j(x_0) 
\]
Hence, the identical calculation in \cite{HIKRigidity} used to compute $\text{Tr}(G)$ applies to compute $\text{Tr}(\tilde G)$ where the only difference is that $c'(x)$ appears in each integral in the calculations there. This yields that near a singularity $T$ of $\text{Tr} 
 \tilde G(t)$, if $F_T$ denotes the set of equivalence classes of periodic broken rays of length $T$,
\[
\frac{d}{ds}|_{s=0}
\text{Tr}\ \p_t G_s(t) \sim -\frac{t}{2} \sum_{[\gamma] \in F_T} C_{[\gamma]} \left(\int_\gamma c' \ ds \right)(t-T+i0)^{-5/2}
\]
modulo lower order singularities, where $C_\gamma$ is the nonzero coefficient appearing in front of $(t-T_{[\gamma]}+i0)^{-5/2}$ in \cite[proposition 2.3]{HIKRigidity} that does not depend on $c'$ nor $\rho'$, and without the factor $T^\sharp_\gamma/N_\gamma$. As described in \cite{HIKRigidity}, the assumptions on $c_0$ ensure that the sum above does not vanish. Since 
\[
\int_\gamma c' \ ds = \int_0^T \frac{d}{ds}|_{s=0} c^{-2}(\gamma(t)) \ dt, 
\]
the earlier proof shows that $c' =0 $ when $\frac{d}{ds}|_{s=0}
\text{Tr} \p_t G_s(t) = 0$.
\end{proof}

\bibliography{poisson_summation}
\bibliographystyle{plain}

\end{document}